\documentclass{amsart}

\newtheorem{theorem}{Theorem}[section]
\newtheorem{lemma}[theorem]{Lemma}

\theoremstyle{definition}

\usepackage{amssymb,latexsym}
\usepackage{enumerate}

\theoremstyle{remark}

\numberwithin{equation}{section}

\begin{document}

\baselineskip=16pt

\title[Weighted stationary phase]
{Weighted stationary phase of higher orders}

\author{Mark McKee \and Haiwei Sun \and Yangbo Ye}

\address{Mark McKee ${}^1$: mark-mckee@uiowa.edu}

\address{Haiwei Sun $\dagger\,{}^{2}$: hwsun@sdu.edu.cn}

\address{Yangbo Ye ${}^{1,2}$: yangbo-ye@uiowa.edu}

\address{${}^1$
Department of Mathematics, The University of Iowa,
Iowa City, Iowa 52242-1419, United States}

\address{${}^2$
School of Mathematics and Statistics, Shandong University,
Weihai, Shandong 264209, China}

\address{$\dagger$
Corresponding author}



\begin{abstract}
The subject matter of this paper is an integral with
exponential oscillation of phase $f(x)$ weighted by $g(x)$ 
on a finite interval $[\alpha,\beta]$. When the phase $f(x)$ 
has a single stationary point in $(\alpha,\beta)$, an 
$n$th-order asymptotic expansion of this integral is proved 
for $n\geq2$. This asymptotic expansion sharpens the classical 
result for $n=1$ by M.N. Huxley. A similar asymptotic expansion 
was proved by Blomer, Khan and Young under the assumptions 
that $f(x)$ and $g(x)$ are smooth and $g(x)$ is compactly 
supported on $\mathbb R$. In the present paper, however, 
these functions are only assumed to be continuously 
differentiable on $[\alpha,\beta]$ $2n+3$ and $2n+1$ times, 
respectively. Because there are no requirements on 
the vanishing of $g(x)$ and its derivatives at the 
endpoints $\alpha$ and $\beta$, the present asymptotic 
expansion contains explicit boundary terms in the main and 
error terms. The asymptotic expansion in this paper 
is thus applicable to a wider class of problems in 
analysis, analytic number theory and other fields.
\end{abstract}

\maketitle

\section{Introduction}

In this paper we will consider exponential integrals
of the form
\begin{eqnarray} \label{Ige(f)}
&&
\int_{\alpha}^{\beta} g(x)e(f(x))~dx.
\end{eqnarray}
When $f'(x)$ changes signs
at a point $x=\gamma$ with $\alpha<\gamma<\beta$,
Huxley \cite{Hxly} obtained a first-order asymptotic
expansion of \eqref{Ige(f)}. This asymptotic
expansion has been widely used as a standard technique in
analytic number theory.
This integral also plays an important role in harmonic 
analysis. In the case of $\alpha=-\infty$ and 
$\beta=\infty$, Walff \cite{Wlff}, pp.38--39, proved an 
$n$th order asymptotic expansion of \eqref{Ige(f)}. 
Blomer, Khan and Young \cite{BlmKhnYng} reproved such an
asymptotic expansion and computed the main terms. 

What we will do in the present paper is to further 
refine the asymptotic expansion of \eqref{Ige(f)} in two 
aspects. Firstly we will consider the case of finite 
lower and upper limits in \eqref{Ige(f)} with 
$g(x)$ and its derivatives being not necessarily zeros 
at the endpoints of the integration interval, as in 
\cite{Hxly}. This will bring in boundary terms which will 
appear both in the main terms and the error terms. 
Detailed treatment of these boundary terms is lengthy, 
but they are necessary for a wider class of applications. 
Secondly the functions $f(x)$ and $g(x)$ will not be assumed 
to be $C^\infty$, as opposite to \cite{Wlff} and 
\cite{BlmKhnYng}. 

Now let us have an overview of the stationary phase 
expansion we will prove (Theorem \ref{WSPI}):
\begin{eqnarray*}
\int_{\alpha}^{\beta}g(x)e( f(x))~dx
&=&
\frac{e(f(\gamma)\pm1/8)}{\sqrt{|f''(\gamma)|}}
\Big(g(\gamma)+\sum_{j=1}^{n}\varpi_{2j}\frac{(-1)^{j}(2j-1)!!}{(2\pi i f''(\gamma))^j}
\Big)
\\
&&
+  \text{Boundary terms} + \text{Error terms}.
\end{eqnarray*}
Here $n$ is related to the smoothness of $f$ and
$g$, $\gamma$ is the only zero of $f'(x)$ in $(\alpha,\beta)$, and $\varpi_{2j}$ are given by \eqref{varpi-5}. 
Possible applications of our results include Salazar and
Ye \cite{SlzYe} on spectral square moments of
$$
S_X(f;\alpha,\beta)
= \sum_n \lambda_f(n)  e(\alpha n^\beta)\phi\Big(\frac{n}{X}\Big)
$$
for $0<\beta<1$, $\alpha \in \mathbb R^\times$,
$\phi \in C_c^\infty ((1,2))$, and $f$ being a Maass form for
$\Gamma_0(N)$, and McKee, Sun and Ye \cite{McKSunYe} on
an improved subconvexity bound for a Rankin-Selberg
$L$-function for $SL_2(\mathbb Z)$ and
$SL_3(\mathbb Z)$ Maass forms.

Our first theorem is a weighted first derivative
test, which strengthens Lemma 5.5.5 of \cite{Hxly},
p.113, with more boundary terms and smaller error terms. 
Similar theorems have been proved and used by 
Jutila and Motohashi \cite{JtlMth} (Lemma 6) and 
Blomer, Khan and Young \cite{BlmKhnYng} (Lemma 8.1). 
We will thus not give its proof here but note that 
our version is on a finite integration interval and comes 
with boundary terms. We will also need the specific form 
of error terms later. 

\begin{theorem}\label{WFDT}
Let $f(x)$ be a real-valued function, $n+2$ times continuously
differentiable for $\alpha\leq x\leq\beta$, and let $g(x)$ be
a real-valued function, $n+1$ times continuously differentiable
for $\alpha\leq x\leq\beta$. Suppose that there are positive parameters $M$,
$N$, $T$, $U$, with $M\geq\beta-\alpha$, and positive constants $C_r$ such that for $\alpha\leq x\leq\beta$,
\begin{equation*}
|f^{(r)}(x)|\leq C_r\frac{T}{M^r},\ |g^{(s)}(x)|\leq C_s\frac{U}{N^s},
\end{equation*}
for $r=2,\ldots,n+2$, and $s=0,\ldots,n+1$. If $f^\prime(x)$ and $f^{\prime\prime}(x)$ do not change signs on the interval $[\alpha,\beta]$,
we have
\begin{eqnarray*}
\int_{\alpha}^{\beta}g(x)e(f(x))dx&=&\Big[e(f(x))\sum_{i=1}^{n}H_{i}(x)\Big]_{\alpha}^{\beta}
\nonumber
+\label{error1}
O\Big(\frac{M}{N}\sum_{j=1}^{[n/2]}\frac{UT^j
}{\min|f^\prime|^{n+j+1}M^{2j}}\sum_{t=j}^{n-j}\frac{1}{N^{n-j-t}M^t}\Big)
\\
&&+\label{error2}
O\Big(\Big(\frac{M}{N}+1\Big)\frac{U}{N^{n}\min{|f^\prime|^{n+1}}}\Big)
+\label{error3}
O\Big(\sum_{j=1}^{n}\frac{UT^j}{\min|f^\prime|^{n+j+1}M^{2j}}\sum_{t=0}^{n-j}\frac{1}{N^{n-j-t}M^{t}}\Big),
\end{eqnarray*}
where
\begin{equation}\label{the2-1}
H_1(x)=\frac{g(x)}{2\pi if^\prime(x)},\ H_i(x)=-\frac{H_{i-1}^\prime(x)}{2\pi if^\prime(x)}
\end{equation}
for $i=2,\ldots,n$.
\end{theorem}

If $g(x)\equiv1$ on $[\alpha,\beta]$, we may take $U=1$ and $N$
arbitrarily large. Then the first two error terms in Theorem \ref{WFDT} are negligible, while in the
third error term we may take only one term with
$t=n-j$ in the inner sum. This way we can get an explicit first
derivative test, which supersedes Lemma 5.5.1 of \cite{Hxly},
p.104, with more boundary terms and smaller error terms.

Our next theorem is an $n$th-order asymptotic expansion
of a weighted stationary phase integral.

\begin{theorem}\label{WSPI}
Let $f(x)$ be a real-valued function, $2n+3$ times continuously
differentiable for $\alpha\leq x\leq\beta$, and $g(x)$ a
real-valued
function, $2n+1$ times continuously differentiable for $\alpha\leq x\leq\beta$.
Let  $H_k(x)$ be defined as in \eqref{the2-1}.
Assume that there are positive parameters $M$,
$N$, $T$, $U$ with
\begin{equation}\label{M>b-a}
M\geq \beta-\alpha,
\end{equation}
and positive constants $C_r$ such that for $\alpha\leq x\leq\beta$,
\begin{equation}\label{U-f'}
|f^{(r)}(x)|\leq C_r\frac{T}{M^r},\ \text{for}\ r=2,\ldots,2n+3,
\end{equation}
\begin{equation}\label{L-f''}
f^{\prime\prime}(x)\geq \frac{T}{C_2M^2}
\end{equation}
and
\begin{equation}\label{U-g}
|g^{(s)}(x)|\leq C_s\frac{U}{N^s},\ \text{for}\ s=0,\ldots,2n+1.
\end{equation}
Suppose that $f'(x)$ changes signs only at $x=\gamma$, from negative to positive, with $\alpha<\gamma<\beta$. 
Let
\begin{equation}\label{Deltamin}
\Delta=
\min\Big\{\frac{\log2}{C_2},\frac{1}{C_2^2\max\limits_{2\leq k\leq 2n+3}\{C_k\}}\Big\}.
\end{equation}
If $T$ is sufficiently large satisfying $T^{\frac{1}{2n+3}}\Delta>1$, we have for $n\geq2$ that
\begin{eqnarray}\label{thm3}
&&
\int_{\alpha}^{\beta}g(x)e(f(x))dx
\\
&=&
\frac{e\Big(f(\gamma)+\frac{1}{8}\Big)}{\sqrt{f''(\gamma)}}\Big(g(\gamma)+\sum_{j=1}^{n}\varpi_{2j}\frac{(-1)^{j}(2j-1)!!}{(4\pi i\lambda_2)^j}
\Big)
+\Big[e(f(x)) \sum_{i=1}^{n+1}H_{i}(x)\Big]_{\alpha}^{\beta}\nonumber
\\
&&+
O\Big(\frac{UM^{2n+5}}{T^{n+2}N^{n+2}}\Big(\frac{1}{(\gamma-\alpha)^{n+2}}+\frac{1}{(\beta-\gamma)^{n+2}}\Big)\Big)\nonumber
+
O\Big(\frac{UM^{2n+4}}{T^{n+2}}
\Big(\frac{1}{(\gamma-\alpha)^{2n+3}}+\frac{1}{(\beta-\gamma)^{2n+3}}\Big)\Big)\nonumber
\\
&&+O\Big(\frac{UM^{2n+4}}{T^{n+2}N^{2n}}\Big(\frac{1}{(\gamma-\alpha)^{3}}+\frac{1}{(\beta-\gamma)^{3}}\Big)\Big)+
O\Big(\frac{U}{T^{n+1}}\Big(\frac{M^{2n+2}}{N^{2n+1}}+M\Big)\Big)
\nonumber
\end{eqnarray}
where
\begin{align}\label{lam}
\lambda_k=\frac{f^{(k)}(\gamma)}{k!}\ \text{for}\ k=2,\ldots,2n+2,
\end{align}
\begin{align}\label{eta2}
\eta_k=\frac{g^{(k)}(\gamma)}{k!}\ \text{for}\ k=0,\ldots,2n,
\end{align}
and
\begin{equation}\label{varpi-5}
\varpi_k
=\eta_k
+
\sum_{\ell=0}^{k-1}
\eta_\ell
\sum_{j=1}^{k-\ell}
\frac{C_{k\ell j}}{\lambda_{2}^j}
\sum_{{\mbox{\tiny$\begin{array}{c}
3\leq n_1,\ldots,n_j\leq 2n+3\\n_1+\cdots+n_j=k-\ell+2j \end{array}$}}}\lambda_{n_1}\cdots\lambda_{n_j},
\end{equation}
with $C_{k\ell j}$ being some constant coefficients.
\end{theorem}

{\it Remark 1.}  Theorem 1.2 also holds when $f'(x)$ changes 
signs from positive to negative, by changing the sign of $1/8$ 
and taking the absolute value 
of $f''(\gamma)$ inside the radical sign on the right hand side 
of \eqref{thm3}. 

{\it Remark 2.} Blomer, Khan and Young's Proposition 8.2 
and Corollary 8.3 in \cite{BlmKhnYng} obtained the 
same main terms and the last big-$O$ term as in 
\eqref{thm3}, under the assymptions that $f(x)$ and $g(x)$ 
are smooth and $g(x)$ is compactly supported on $\mathbb R$. 
In Theorem \ref{WSPI}, however, $f(x)$ and $g(x)$ are only 
assumed to be continuously differentiable on $[\alpha,\beta]$ 
$2n+3$ and $2n+1$ times, respectively. Because there are 
no requirements that $g(x)$ and its derivatives vanish at 
the endpoints $\alpha$ and $\beta$, Theorem \ref{WSPI} 
is valid for a much wider class of functions $f(x)$ and $g(x)$. 
This is indeed the case in \cite{SlzYe}. 

{\it Remark 3.} 
In \cite{BlmKhnYng} 
the parameters require a condition that $N\leq\beta-\alpha$. 
In Theorem \ref{WSPI} \eqref{M>b-a} is assumed instead, which 
is the same as assumed in Huxley \cite{Hxly}. 

We end this introduction with an outline of the proof. 
In \S2 we divide the integration interval into three parts: $[\alpha, u]$, $[u,v]$ and $[v,\beta]$. In the middle 
subinterval we change variables from $x$ to $y$ by  
$f(x) - f(\gamma) = \lambda_2 y^2$. 
The goal is to obtain a Taylor approximation (in $y$) for 
$g(x) \frac{dx}{dy}$ and its $y$ derivatives.  
The Taylor approximation (without error) to $g(x) \frac{dx}{dy}$ is given by $\sum_{k=0}^{2n} \varpi_{k} y^k$, 
where $\varpi_k$ is given by (\ref{varpi-5}) above. 

In \S3, we estimate the main weighted stationary phase integral. 
The main term comes from integrating the middle range in $y$.
$x \in [u,v]$ corresponds to $y \in [-r,r]$ for $r$ given by (\ref{range-r}). 
This leads us to estimate integrals of the form
\[
\int_{-r}^r e \left(  \lambda_2 y^2  \right)  y^k ~dy,
\]
(for even $k$) which we do by an application of the probability integral
\[
\Phi(x) =  \frac{1}{\sqrt{x}} \int_0^x \frac{e^{-t}}{\sqrt{t}} ~dt,
\]
and estimates of this integral in Gradshteyn
and Ryzhik \cite{GradRzhk}.
We need to estimate the integral of the error 
\begin{equation}\label{Qg-y}
Q(y) = g(x) \frac{dx}{dy} - \sum_{k=0}^{2n} \varpi_{k} y^k,
\end{equation}
and its $y$ derivatives.  This involves previous estimates and the second derivative test, found in Huxley \cite{Hxly}.
Also used is a dyadic decomposition of the interval $ [-r,r]$.

Our proof is different from those in \cite{Wlff} and 
\cite{BlmKhnYng}. The reason we need to cut $[\alpha,\beta]$ 
into three pieces is that our techniques can only be applied to 
a small neighborhood of $\gamma$. On outer subintervals 
$[\alpha,u]$ and $[v,\beta]$ we apply Theorem \ref{WFDT}.

\section{ Lemmas for Theorem \ref{WSPI}}

Under the assumptions of Theorem \ref{WSPI}, we have the following
Taylor expansions at $x=\gamma$,
\begin{align}\label{Taylor-f}
f(x)=f(\gamma)+\sum_{k=2}^{2n+2}\lambda_k(x-\gamma)^k
+
\frac{f^{(2n+3)}(\eta_0)}{(2n+3)!}
(x-\gamma)^{2n+3},
\end{align}
and for $1\leq i\leq 2n+2$
\begin{align}\label{Taylor-f-i}
f^{(i)}(x)=\sum_{k=\max\{2,i\}}^{2n+2}
\lambda_k^{(i)}(x-\gamma)^{k-i}+
\frac{f^{(2n+3)}(\eta_i)}{(2n+3-i)!}
(x-\gamma)^{2n+3-i},
\end{align}
where $\eta_0,\ldots,\eta_{2n+2}$ are numbers between $x$ and $\gamma$ depending on $x$. Here
\begin{align*}\label{lam-i}
\lambda_k^{(i)}=\frac{f^{(k)}(\gamma)}{(k-i)!}=\frac{k!}{(k-i)!}\lambda_k\ \text{for}\ \max\{2,i\}\leq k\leq2n+2.
\end{align*}

Now we change variables from $x$ to $y=h(x-\gamma)$ by
\begin{equation}\label{thm3-1}
f(x)-f(\gamma)=\lambda_2(h(x-\gamma))^2=\lambda_2y^2,
\end{equation}
such that $y=h(x-\gamma)$ has the same sign as that of
$x-\gamma$. 
Define $h(\alpha-\gamma)=-r_1$ and $h(\beta-\gamma)=r_2$, 
i.e., $f(\alpha)=f(\gamma)+\lambda_2r_1^2$ and 
$f(\beta)=f(\gamma)+\lambda_2r_2^2$.
We choose a number $r$ which satisfies that
\begin{align}\label{range-r}
r=\min\{r_1,r_2,\Delta M\}.
\end{align} 
Since
\begin{equation*}
\lambda_2r_1^2
=
f(\alpha)-f(\gamma)
=\frac{f^{\prime\prime}(\eta)}{2!}|\alpha-\gamma|^2
\end{equation*}
for some $\eta\in(\alpha,\gamma)$, by \eqref{U-f'} and \eqref{L-f''}
we have
\begin{eqnarray*}
\frac{1}{C_2^2}\leq \frac{|\alpha-\gamma|^2}{r_1^2}=\frac{2\lambda_2}{f^{\prime\prime}(\eta)}\leq C_2^2.
\end{eqnarray*}
Hence we have
\begin{eqnarray}\label{r1-range}
\frac{r_1}{C_2}\leq |\alpha-\gamma|\leq C_2r_1.
\end{eqnarray}
Likewise
\begin{eqnarray}\label{r2-range}
\frac{r_2}{C_2}\leq |\beta-\gamma|\leq C_2r_2.
\end{eqnarray}
By \eqref{range-r}, \eqref{r1-range} and \eqref{r2-range}, we see that
\begin{eqnarray}\label{estimate-r}
r\geq \min\Big\{\frac{|\alpha-\gamma|}{C_2},\frac{|\beta-\gamma|}{C_2},\frac{M}{T^{\frac{1}{2n+3}}}\Big\}
\end{eqnarray}

Define $u,v$ by $h(u-\gamma)=-r$ and $h(v-\gamma)=r$, 
i.e., $f(u)=f(v)=f(\gamma)+\lambda_2r^2$.
By \eqref{range-r} we see that 
$\alpha\leq u< \gamma< v\leq \beta$.
In this section, we will only consider $x$ and $y$ in
\begin{eqnarray}\label{range-x and y}
-r\leq y\leq r,\ u\leq x\leq v.
\end{eqnarray}
By \eqref{Taylor-f} we know that
\begin{align*}
y^2
&=
(x-\gamma)^2\Big(1+\sum_{k=3}^{2n+2}\frac{\lambda_k}{\lambda_2}(x-\gamma)^{k-2}+
\frac{f^{(2n+3)}(\eta_0)}{\lambda_2(2n+3)!}
(x-\gamma)^{2n+1}
\Big).
\end{align*}
Similar to \eqref{r1-range}, we have
\begin{eqnarray}\label{u,v-range}
\frac{r}{C_2}\leq |u-\gamma|\leq C_2r\ \text{and}\ \frac{r}{C_2}\leq |v-\gamma|\leq C_2r.
\end{eqnarray}
By \eqref{range-r}, \eqref{range-x and y} and \eqref{u,v-range}, we see that
\begin{eqnarray*}
|x-\gamma|\leq \max(|u-\gamma|, |v-\gamma|)
\leq C_2r\leq C_2\Delta M.
\end{eqnarray*}
Note that by \eqref{U-f'}, \eqref{L-f''}  and \eqref{lam} we have
$$
\frac{|\lambda_{k}|}{\lambda_2}
\leq \frac{2C_{k}C_2}{k!M^{k-2}}.
$$
Therefore we have
\begin{eqnarray}\label{lessthan1}
&&
\Big|\sum_{k=3}^{2n+2}
\frac{\lambda_k}{\lambda_2}(x-\gamma)^{k-2}+
\frac{f^{(2n+3)}(\eta_0)}{\lambda_2(2n+3)!}
(x-\gamma)^{2n+1}
\Big|
\\
&\leq&
\sum_{k=3}^{2n+3}\frac{2C_{k}C_2(C_2\Delta)^{k-2}}{k!}
<
\frac{C_2^2\max\{C_k\}\Delta}3
\sum_{k=0}^{2n}\frac{(C_2\Delta)^k}{k!}
\nonumber
\\
&<&
\frac{C_2^2\max\{C_k\}\Delta}3
e^{C_2\Delta}
<
C_2^2\max\{C_k\}\Delta
\leq1
\nonumber
\end{eqnarray}
because of \eqref{Deltamin}.
By \eqref{lessthan1} we can get for $1\leq j\leq 2n+1$
\begin{eqnarray}\label{y^j}
y^j
&=&
(x-\gamma)^j\Big(1+\sum_{k=1}^{2n}\frac{\lambda_{k+2}}{\lambda_2}(x-\gamma)^{k}+
\frac{f^{(2n+3)}(\eta_0)}{\lambda_2(2n+3)!}
(x-\gamma)^{2n+1}
\Big)^{j/2}
\\
&=&
(x-\gamma)^j\Big(1+\sum_{k=1}^{2n}\mu_{jk}(x-\gamma)^{k}+O_n\Big(\frac{|x-\gamma|^{2n+1}}{M^{2n+1}}\Big)\Big)\nonumber
\end{eqnarray}
by Taylor expansion at $x=\gamma$. Here the Taylor coefficients
$\mu_{jk}$ can be determined by applying binomial
expansions to \eqref{y^j}
\begin{align*}
y^j&=(x-\gamma)^j\Big(1+\sum_{i=1}^{\infty}C_{ji}
\Big(\sum_{k=1}^{2n}\frac{\lambda_{k+2}}{\lambda_2}(x-\gamma)^{k}+\frac{f^{(2n+3)}(\eta_0)}{\lambda_2(2n+3)!}
(x-\gamma)^{2n+1}\Big)^{i}\Big)
\end{align*}
with 
$$
C_{ji}
=\Big(\mbox{$\begin{array}{c}j/2\\i\end{array}$}\Big) 
= \frac1{i!}
\prod_{m=0}^{i-1} \Big(\frac{j}{2}-m\Big). 
$$
Consequently
\begin{align}\label{mu_k}
\mu_{j0}=1,\ \mu_{jk}= \sum_{i=1}^{k}\frac{C_{ji}}{\lambda_2^i}\sum_{{\mbox{\tiny$\begin{array}{c}
3\leq n_1,\ldots,n_i\leq 2n+3\\ n_1+\cdots+n_i=k+2i\end{array}$}}}\lambda_{n_1}\cdots\lambda_{n_i}\ \text{for}\ 1\leq k\leq 2n.
\end{align}

The variable change between $x$ and $y$ in \eqref{thm3-1} and
\eqref{y^j}  allows us to express $f^{(i)}(x)$
in terms of $y$ for $1\leq i\leq 2n+2$.

\begin{lemma}\label{lem-1-1}
Suppose \eqref{U-f'} and \eqref{L-f''} hold for $f(x)$.
For $x$ and $y$ in \eqref{range-x and y} with $r$ in
\eqref{range-r} we have
\begin{align*}
f'(x)&=\sum_{k=1}^{2n+1}\theta_{k}^{(1)}y^{k}+O_n\Big(\frac{T|y|^{2n+2}}{M^{2n+3}}\Big),
\end{align*}
where
\begin{align*}
\theta_{1}^{(1)}=\lambda_2^{(1)}=2\lambda_2,\
\theta_{k}^{(1)}=\sum_{ j=1}^{k-1}\frac{C^{(1)}_{k,j}}{\lambda_2^{j-1}}\sum_{{\mbox{\tiny$\begin{array}{c}
3\leq n_1,\ldots,n_j\leq 2n+3\\ n_1+\cdots+n_j=k-1+2j\end{array}$}}}\lambda_{n_1}\cdots\lambda_{n_j}\ \text{for} \ 2\leq k\leq 2n+1.
\end{align*}
\end{lemma}

\begin{proof} We claim that for $1\leq m\leq 2n+1$
\begin{align}\label{f'(x)F1}
f'(x)&=\sum_{k=1}^{m}\theta_{m,k}y^{k}+\sum_{k=m+1}^{2n+1}\theta_{m,k}(x-\gamma)^{k}+O_n\Big(\frac{T|x-\gamma|^{2n+2}}{M^{2n+3}}\Big),
\end{align}
where
\begin{align*}
\theta_{m,1}=\lambda_2^{(1)}=2\lambda_2,\
\theta_{m,k}=\sum_{1\leq j\leq k-1}\frac{C_{k,j}}{\lambda_2^{j-1}}\sum_{{\mbox{\tiny$\begin{array}{c}
n_1+\cdots+n_j=k-1+2j\\ 3\leq n_1,\ldots,n_j\leq 2n+3\end{array}$}}}\lambda_{n_1}\cdots\lambda_{n_j}\ \text{for}\ 2\leq k\leq 2n+1,
\end{align*}
which can be proved by induction.
Take $m=2n+1$ in  \eqref{f'(x)F1} we get
\begin{align*}
f'(x)&=\sum_{k=1}^{2n+1}\theta_{2n+1,k}y^{k}+O_n\Big(\frac{T|x-\gamma|^{2n+2}}{M^{2n+3}}\Big)
=:\sum_{k=1}^{2n+1}\theta_{k}^{(1)}y^{k}+O_n\Big(\frac{T|x-\gamma|^{2n+2}}{M^{2n+3}}\Big).
\end{align*}
Using \eqref{thm3-1} and from the second  order Taylor expansion we see that
\begin{align*}
\lambda_2 y^2=f(x)-f(\gamma)=\frac{f^{\prime\prime}(w)}{2!}(x-\gamma)^2,
\end{align*}
where $w$ is some constant between $x$ and $\gamma$.  Then by \eqref{U-f'} and \eqref{L-f''}
\begin{eqnarray*}
\frac{1}{C_2^2}\leq \frac{|x-\gamma|^2}{y^2}=2\frac{\lambda_2}{f^{\prime\prime}(\eta)}\leq C_2^2.
\end{eqnarray*}
Hence similar to \eqref{u,v-range} we get
$|x-\gamma|/C_2\leq y\leq C_2|x-\gamma|$. 
Then using above estimates we get
\begin{align*}
f'(x)&=\sum_{k=1}^{2n+1}\theta_{k}^{(1)}y^{k}+O_n\Big(\frac{T|y|^{2n+2}}{M^{2n+3}}\Big).
\end{align*}
\end{proof}

Similarly, we can change $x$ to $y$ in \eqref{Taylor-f-i} by \eqref{y^j}. We have
for $2\leq i\leq 2n+2$
\begin{align}\label{f^i(x)-2}
f^{(i)}(x)&=\sum_{k=0}^{2n+2-i}\theta_{k}^{(i)}y^{k}+O_n\Big(\frac{T|y|^{2n+3-i}}{M^{2n+3}}\Big),
\end{align}
where
\begin{align*}
\theta_{0}^{(i)}=\lambda_i^{(i)}=i!\lambda_i,\
\theta_{k}^{(i)}=\sum_{ j=1}^{k}\frac{C_{k,j}}{\lambda_2^{j-1}}\sum_{{\mbox{\tiny$\begin{array}{c}
3\leq n_1,\ldots,n_j\leq 2n+3\\ n_1+\cdots+n_j=k+i-2+2j\end{array}$}}}\lambda_{n_1}\cdots\lambda_{n_j}\ \text{for}\ 1\leq k\leq 2n+2-i.
\end{align*}
Now by the definition  of $y$ in \eqref{thm3-1} we can compute $\frac{dx}{dy}$.

\begin{lemma}\label{lem-1}
With the above notation we assume \eqref{U-f'} and \eqref{L-f''}. Then
\begin{align}\label{dx/dy}
\frac{dx}{dy}&=\sum_{k=0}^{2n}\rho_ky^{k}+O_n\Big(\frac{|y|^{2n+1}}{M^{2n+1}}\Big)
\end{align}
where
\begin{align}\label{rho}
\rho_0=1,\ \rho_k= \sum_{j=1}^{k}\frac{C'_{kj}}{\lambda_2^{j}}\sum_{{\mbox{\tiny$\begin{array}{c}
3\leq n_1,\ldots,n_j\leq 2n+3\\n_1+\cdots+n_j=k+2j \end{array}$}}}\lambda_{n_1}\cdots\lambda_{n_j}\ \text{for}\ k\geq 1.
\end{align}
\end{lemma}

\begin{proof} First by \eqref{thm3-1} we see that
\begin{align}\label{dx/dy-1}
\frac{dx}{dy}=\frac{2\lambda_2y}{f'(x)}.
\end{align}
Then by \eqref{Taylor-f-i} with $i=1$  we get
\begin{eqnarray*}
\frac{dx}{dy}=\frac{2\lambda_2y}{2\lambda_2(x-\gamma)\Big(1+\sum_{k=3}^{2n+2}
\frac{k\lambda_k}{2\lambda_2}(x-\gamma)^{k-2}+
\frac{f^{(2n+3)}(\eta_1)}{2\lambda_2(2n+2)!}
(x-\gamma)^{2n+1}\Big)}.
\end{eqnarray*}
Similar to \eqref{lessthan1}, we can prove that
\begin{eqnarray*}
&&
\Big|\sum_{k=3}^{2n+2}
\frac{k\lambda_k}{2\lambda_2}(x-\gamma)^{k-2}+
\frac{f^{(2n+3)}(\eta_1)}{2\lambda_2(2n+2)!}
(x-\gamma)^{2n+1}\Big|
\\
&\leq&
\sum_{k=3}^{2n+3}\frac{C_{k}C_2(C_2\Delta)^{k-2}}{(k-1)!}
<
\frac{C_2^2\max\{C_k\}\Delta}2
\sum_{k=0}^{2n}\frac{(C_2\Delta)^k}{k!}
\nonumber
\\
&<&
\frac{C_2^2\max\{C_k\}\Delta}2
e^{C_2\Delta}
\leq 
C_2^2\max\{C_k\}\Delta
\leq1
\end{eqnarray*}
because of \eqref{Deltamin}. Therefore we get 
\begin{eqnarray}\label{backwards-1}
\frac{dx}{dy}&=&\frac{y}{x-\gamma}\sum_{j=0}^{\infty}(-1)^j\Big(\sum_{k=3}^{2n+2}
\frac{k\lambda_k}{2\lambda_2}(x-\gamma)^{k-2}+
\frac{f^{(2n+3)}(\eta_1)}{2\lambda_2(2n+2)!}
(x-\gamma)^{2n+1}\Big)^j
\\
&=&\frac{y}{x-\gamma}\Big(1+\sum_{k=1}^{2n}\mu_{k}'(x-\gamma)^{k}+O_n\Big(\frac{|x-\gamma|^{2n+1}}{M^{2n+1}}\Big)\Big)\nonumber
\end{eqnarray}
with 
\begin{align*}
\mu_{0}'=1,\ \mu_{k}'= \sum_{i=1}^{k}\frac{C_{i}'}{\lambda_2^i}\sum_{{\mbox{\tiny$\begin{array}{c}
3\leq n_1,\ldots,n_i\leq 2n+3\\ n_1+\cdots+n_i=k+2i\end{array}$}}}\lambda_{n_1}\cdots\lambda_{n_i}\ \text{for}\ 1\leq k\leq 2n.
\end{align*}

Now by \eqref{y^j}  with $j=1$, we have (noting \eqref{lessthan1})
\begin{eqnarray}\label{backwards-2}
\frac{y}{x-\gamma}&=&\Big(1+\sum_{k=1}^{2n}\frac{\lambda_{k+2}}{\lambda_2}(x-\gamma)^{k}+
\frac{f^{(2n+3)}(\eta_0)}{\lambda_2(2n+3)!}
(x-\gamma)^{2n+1}
\Big)^{1/2}
\\
&=&1+\sum_{k=1}^{2n}\mu_{k}''(x-\gamma)^{k}+O_n\Big(\frac{|x-\gamma|^{2n+1}}{M^{2n+1}}\Big)\nonumber
\end{eqnarray}
with
\begin{align*}
\mu_{0}''=1,\ \mu_{k}''= \sum_{i=1}^{k}\frac{C_{i}''}{\lambda_2^i}\sum_{{\mbox{\tiny$\begin{array}{c}
3\leq n_1,\ldots,n_i\leq 2n+3\\ n_1+\cdots+n_i=k+2i\end{array}$}}}\lambda_{n_1}\cdots\lambda_{n_i}\ \text{for}\ 1\leq k\leq 2n.
\end{align*}

Then  by \eqref{backwards-1} and \eqref{backwards-2}, we conclude that
\begin{eqnarray}\label{dx/dy-new-1}
\frac{dx}{dy}&=&\Big(1+\sum_{k=1}^{2n}\mu_{k}'(x-\gamma)^{k}+O_n\Big(\frac{|x-\gamma|^{2n+1}}{M^{2n+1}}\Big)\Big)
\Big(1+\sum_{k=1}^{2n}\mu_{k}''(x-\gamma)^{k}+O_n\Big(\frac{|x-\gamma|^{2n+1}}{M^{2n+1}}\Big)\Big)\nonumber
\\
&=&\sum_{k=0}^{2n}\mu_{k}'''(x-\gamma)^{k}+O_n\Big(\frac{|x-\gamma|^{2n+1}}{M^{2n+1}}\Big),
\end{eqnarray}
with $\mu_{k}'''=\sum_{i=0}^{k}\mu_{i}'\mu_{k-i}''$. 
Since (ignoring all coefficients)
\begin{eqnarray*}
\mu_{k}'\mu_{l}''&=&\sum_{1\leq s\leq k}\frac{1}{\lambda_2^{s}}\sum_{{\mbox{\tiny$\begin{array}{c}
3\leq n_i\leq 2n+3\\ n_1+\cdots+n_s=k+2s\end{array}$}}}\lambda_{n_1}\cdots\lambda_{n_s}
\\
&&\times
\sum_{1\leq s\leq l}\frac{1}{\lambda_2^{s}}\sum_{{\mbox{\tiny$\begin{array}{c}
3\leq n_i\leq 2n+3\\ n_1+\cdots+n_s=l+2s\end{array}$}}}\lambda_{n_1}\cdots\lambda_{n_s}
\\
&=&
\sum_{1\leq s\leq k+l}\frac{1}{\lambda_2^{s}}\sum_{{\mbox{\tiny$\begin{array}{c}
3\leq n_i\leq 2n+3\\ n_1+\cdots+n_s=k+l+2s\end{array}$}}}\lambda_{n_1}\cdots\lambda_{n_s},
\end{eqnarray*}
we get that
\begin{align*}
\mu_{0}'''=1,\ \mu_{k}'''= \sum_{i=1}^{k}\frac{C_{i}'''}{\lambda_2^i}\sum_{{\mbox{\tiny$\begin{array}{c}
3\leq n_1,\ldots,n_i\leq 2n+3\\ n_1+\cdots+n_i=k+2i\end{array}$}}}\lambda_{n_1}\cdots\lambda_{n_i}\ \text{for}\ 1\leq k\leq 2n.
\end{align*}

Finally,  changing $x$ to $y$ in \eqref{dx/dy-new-1}, we get
\begin{align*}
\frac{dx}{dy}&=\sum_{k=0}^{2n}\rho_ky^{k}+O_n\Big(\frac{|y|^{2n+1}}{M^{2n+1}}\Big)
\end{align*}
where
\begin{align*}
\rho_0=1,\ \rho_k= \sum_{j=1}^{k}\frac{C'_{kj}}{\lambda_2^{j}}\sum_{{\mbox{\tiny$\begin{array}{c}
3\leq n_1,\ldots,n_j\leq 2n+3\\n_1+\cdots+n_j=k+2j \end{array}$}}}\lambda_{n_1}\cdots\lambda_{n_j}\ \text{for}\ k\geq 1,
\end{align*}
which can be proved by induction like Lemma \ref{lem-1-1}.
\end{proof}

By the assumptions of Theorem \ref{WSPI}, we also have the following Taylor expansions
\begin{align}\label{g-1}
g(x)&
=\sum_{k=0}^{2n}\eta_k(x-\gamma)^k+O_n\Big(\frac{U|x-\gamma|^{2n+1}}{N^{2n+1}}\Big),
\end{align}
and
\begin{align}\label{g-i}
\frac{d^ig}{dx^i}&
=\sum_{k=i}^{2n}\eta_k^{(i)}(x-\gamma)^{k-i}+O_n\Big(\frac{U|x-\gamma|^{2n+1-i}}{N^{2n+1}}\Big),
\end{align}
with
\begin{align*}
\eta_k^{(i)}=\frac{g^{(k)}(\gamma)}{(k-i)!}=\frac{k!}{(k-i)!}\eta_k\ \text{for}\ i\leq k\leq 2n.
\end{align*}
Similarly, if we change variables in \eqref{g-1} to $y$, we can get
\begin{align}\label{g}
g(x)&
=\sum_{k=0}^{2n}\eta'_ky^k+O_n\Big(U|y|^{2n+1}\Big(\frac{1}{NM^{2n}}+\frac{1}{N^{2n+1}}\Big)\Big),
\ \ \text{with}
\ \ \eta'_0=g(\gamma).
\end{align}
To determine other $\eta'_k$, we substitute \eqref{y^j} into \eqref{g} to get
\begin{eqnarray*}
g(x)
&=&
\eta'_0+\sum_{k=1}^{2n}\eta'_k(x-\gamma)^k
\Big(\sum_{\ell=0}^{2n}\mu_{k\ell}(x-\gamma)^{\ell}+O_n\Big(\frac{|x-\gamma|^{2n+1}}{M^{2n+1}}\Big)\Big)
+
O_n\Big(U|y|^{2n+1}\Big(\frac{1}{NM^{2n}}+\frac{1}{N^{2n+1}}\Big)\Big)
\\
&=&
\eta'_0+\sum_{m=1}^{2n}(x-\gamma)^{m}\sum_{{\mbox{\tiny$\begin{array}{c}
k\geq 1\\ \ell\geq 0\\ k+\ell=m\end{array}$}}}\eta'_k\mu_{k\ell}+O_n\Big(U|y|^{2n+1}\Big(\frac{1}{NM^{2n}}+\frac{1}{N^{2n+1}}\Big)\Big).
\end{eqnarray*}
Consequently
\begin{align}\label{eta-1}
\eta_1=\sum_{{\mbox{\tiny$\begin{array}{c}
k\geq 1\\\ell\geq 0\\ k+\ell=m\end{array}$}}}\eta'_k\mu_{k\ell}=\eta'_1\mu_{10}=\eta'_1,
\end{align}
and for $2\leq m\leq 2n$
\begin{align}\label{eta-m}
\eta_m=\sum_{{\mbox{\tiny$\begin{array}{c}
k\geq 1\\ \ell\geq 0\\ k+\ell=m\end{array}$}}}\eta'_k\mu_{k\ell}=\eta'_m+\sum_{{\mbox{\tiny$\begin{array}{c}
k,\ell\geq 1\\ k+\ell=m\end{array}$}}}\eta'_k\mu_{k\ell},
\end{align}
where $\eta_m$  is defined in \eqref{g-1}. We may compute
$\eta'_m$ for $m\geq 2$ recursively using \eqref{eta-m}.

\begin{lemma}\label{lem-2}
With the above notation
\begin{align}\label{eta'-m}
\eta'_m=\eta_m+\sum_{k=1}^{m-1}\eta_k\sum_{j=1}^{m-k}\frac{C_{mkj}}{\lambda_2^j}
\sum_{{\mbox{\tiny$\begin{array}{c}
3\leq n_1,\ldots,n_j\leq 2n+3\\n_1+\cdots+n_j=m-k+2j \end{array}$}}}\lambda_{n_1}\cdots\lambda_{n_j}
\end{align}
where the $k$ sum vanishes when $m\leq 1$.
\end{lemma}

\begin{proof} By \eqref{eta-1}, \eqref{eta'-m} holds for $m=1$. Suppose \eqref{eta'-m} holds for any number $\leq m$.
Then by \eqref{eta-m} and \eqref{mu_k}
\begin{eqnarray}\label{eta'-m+1}
\eta'_{m+1}
&=&
\eta_{m+1}-\sum_{\ell=1}^{m}\eta'_\ell\mu_{\ell,m+1-\ell}
\\
&=&
\eta_{m+1}-\sum_{\ell=1}^{m}\eta_\ell\sum_{i=1}^{m+1-\ell}\frac{C_{\ell i}}{\lambda_2^i}\sum_{{\mbox{\tiny$\begin{array}{c}
3\leq m_1,\ldots,m_i\leq 2n+3\\ m_1+\cdots+m_i=m+1-\ell+2i\end{array}$}}}\lambda_{m_1}\cdots\lambda_{m_i}
\nonumber
\\
&&-
\sum_{\ell=1}^{m}\sum_{k=1}^{\ell-1}\eta_k\sum_{j=1}^{\ell-k}\frac{C_{\ell kj}}{\lambda_2^j}
\sum_{{\mbox{\tiny$\begin{array}{c}
3\leq n_1,\ldots,n_j\leq 2n+3\\n_1+\cdots+n_j=\ell-k+2j \end{array}$}}}\lambda_{n_1}\cdots\lambda_{n_j}
\nonumber
\\
&&\times
\sum_{i=1}^{m+1-\ell}\frac{C_{\ell i}}{\lambda_2^i}\sum_{{\mbox{\tiny$\begin{array}{c}
3\leq m_1,\ldots,m_i\leq 2n+3\\ m_1+\cdots+m_i=m+1-\ell+2i\end{array}$}}}\lambda_{m_1}\cdots\lambda_{m_i}.
\nonumber
\end{eqnarray}
The first two terms on the right side of \eqref{eta'-m+1} fit \eqref{eta'-m} for $m+1$. For the third term, we change
the order of sums on $\ell$ and $k$, let $h=i+j$, denote $m_1,\ldots,m_i,n_1,\ldots,n_j$ by $p_1,\ldots,p_h$, and get
\begin{align*}
-&\sum_{k=1}^{m-1}\eta_k\sum_{\ell=k+1}^{m}\sum_{h=2}^{m+1-k}\frac{1}{\lambda_2^h}
\sum_{{\mbox{\tiny$\begin{array}{c}
i,j\geq 1\\ i+j=h\end{array}$}}}C_{\ell i}C_{\ell kj}
\sum_{{\mbox{\tiny$\begin{array}{c}
3\leq p_1,\ldots,p_h\leq 2n+3\\p_1+\cdots+p_h=m+1-k+2h\\p_1+\cdots+p_i=m+1-\ell+2i \end{array}$}}}\lambda_{p_1}\cdots\lambda_{p_h} ��
\end{align*}
which also fits \eqref{eta'-m} for $m+1$.
\end{proof}

Similarly using \eqref{y^j} in \eqref{g-i} we get for $1\leq i\leq n+1$
\begin{align}\label{g^i}
\frac{d^ig}{dx^i}&
=\sum_{k=0}^{2n-i}{\eta_k^{(i)}}'y^k+O_n\Big(U|y|^{2n+1-i}\Big(\frac{1}{NM^{2n}}+\frac{1}{N^{2n+1}}\Big)\Big),
\end{align}
where
\begin{align}\label{eta'-m-1}
{\eta_k^{(i)}}'=\frac{(k+i)!}{k!}\eta_{k+i}+\sum_{m=1}^{k+i-1}\eta_m\sum_{j=1}^{k+i-m}\frac{C_{kmj}}{\lambda_2^j}
\sum_{{\mbox{\tiny$\begin{array}{c}
3\leq n_1,\ldots,n_j\leq 2n+3\\n_1+\cdots+n_j=k+i-m+2j \end{array}$}}}\lambda_{n_1}\cdots\lambda_{n_j}.
\end{align}

Multiplying $g(x)$ in \eqref{g} with $\frac{dx}{dy}$ in \eqref{dx/dy} and using
\begin{eqnarray*}
{\eta_k^{(i)}}'\ll\Big(\frac{U}{N^k}+\frac{U}{NM^{k-1}}\Big),\ \rho_k\ll\frac{1}{M^k},
\end{eqnarray*}
we get
\begin{align}\label{g-dx/dy}
g(x)\frac{dx}{dy}=\sum_{k=0}^{2n}\varpi_ky^k+O_n\Big(U|y|^{2n+1}\Big(\frac{1}{M^{2n+1}}+\frac{1}{N^{2n+1}}\Big)\Big),
\end{align}
where
\begin{align}\label{varpi-1}
\varpi_k=\sum_{\ell=0}^{k}\eta'_\ell\rho_{k-\ell}.
\end{align}
Note that by \eqref{rho} and \eqref{g},
$\varpi_0=\eta'_0\rho_{0}=g(\gamma)$.

\begin{lemma}
With the notation as above, \eqref{varpi-5} holds.
\end{lemma}

\begin{proof}
By \eqref{varpi-1}, \eqref{eta'-m} and \eqref{rho} we have
\begin{eqnarray}
\varpi_N
&=&
\sum_{m=0}^{N-1}\Big(\eta_m+\sum_{k=1}^{m-1}\eta_k\sum_{j=1}^{m-k}\frac{C_{mkj}}{\lambda_2^j}\sum_{{\mbox{\tiny$\begin{array}{c}
3\leq n_1,\ldots,n_j\leq 2n+3\\n_1+\cdots+n_j=m-k+2j \end{array}$}}}\lambda_{n_1}\cdots\lambda_{n_j}\Big)\label{var-6}\nonumber
\\
&&\times
\sum_{h=1}^{N-m}\frac{C_{N-m,h}}{\lambda_2^{h}}\sum_{{\mbox{\tiny$\begin{array}{c}
3\leq m_1,\ldots,m_h\leq 2n+3\\m_1+\cdots+m_j=N-m+2h\end{array}$}}}\lambda_{m_1}\cdots\lambda_{m_h}
\label{var-7}
\\
&&+
\eta_N+\sum_{k=1}^{N-1}\eta_k\sum_{j=1}^{N-k}\frac{C_{Nk j}}{\lambda_{2}^j}\sum_{{\mbox{\tiny$\begin{array}{c}
3\leq n_1,\ldots,n_j\leq 2n+3\\n_1+\cdots+n_j=N-k+2j \end{array}$}}}\lambda_{n_1}\cdots\lambda_{n_j},
\label{var-8}
\end{eqnarray}
where the first $k$ sum vanishes when $m\leq 1$. Note that \eqref{var-8} is contained in \eqref{varpi-5},
while $\sum_{m=0}^{N-1}\eta_m$ times \eqref{var-7} is also contained in \eqref{varpi-5}. The rest sum of
$\varpi_N$ equals
\begin{eqnarray*}
&&
\sum_{m=2}^{N-1}\sum_{k=1}^{m-1}\eta_k\sum_{j=1}^{m-k}\frac{C_{mkj}}{\lambda_2^j}\sum_{{\mbox{\tiny$\begin{array}{c}
3\leq n_1,\ldots,n_j\leq 2n+3\\n_1+\cdots+n_j=m-k+2j \end{array}$}}}\lambda_{n_1}\cdots\lambda_{n_j}
\\
&&\times
\sum_{h=1}^{N-m}\frac{C_{N-m,h}}{\lambda_2^{h}}\sum_{{\mbox{\tiny$\begin{array}{c}
3\leq m_1,\ldots,m_h\leq 2n+3\\m_1+\cdots+m_j=N-m+2h\end{array}$}}}\lambda_{m_1}\cdots\lambda_{m_h}
\\
&=&
\sum_{k=1}^{N-2}\eta_k\sum_{m=k+1}^{N-1}\sum_{i=1}^{N-k}\frac{1}{\lambda_{2}^i}
\sum_{{\mbox{\tiny$\begin{array}{c}
h,j\geq1\\h+j=i \end{array}$}}}C_{mkj}C_{N-m,h}
\sum_{{\mbox{\tiny$\begin{array}{c}
3\leq p_1,\ldots,p_i\leq 2n+3\\p_1+\cdots+p_i=N-k+2i\\p_1+\cdots +p_h=N-m+2h \end{array}$}}}\lambda_{p_1}\cdots\lambda_{p_i},
\end{eqnarray*}
which also fits \eqref{varpi-5}.
\end{proof}

In the following lemma we compute derivatives of $g(x)\frac{dx}{dy}$ which we will use to prove our main theorem.

\begin{lemma}\label{lem-2-4}
With the above notation we assume \eqref{U-f'} and \eqref{L-f''}. Then for $1\leq i\leq n+2$
\begin{eqnarray}\label{g-dx-1}
\frac{d^i}{dy^i}\Big(g(x)\frac{dx}{dy}\Big)=\sum_{k=0}^{2n-i}\frac{(k+i)!}{k!}\varpi_{k+i}y^k+O_n\Big(U|y|^{2n+1-i}\Big(\frac{1}{M^{2n+1}}+\frac{1}{N^{2n+1}}\Big)\Big).
\end{eqnarray}
\end{lemma}

\begin{proof}
By \eqref{varpi-5} we see that the expression of $\varpi_{k}$, $0\leq k\leq 2n$, only uses $\eta_{\ell}$, $0\leq \ell\leq 2n$, and hence it only uses $g^{(\ell)}(\gamma)$ for $0\leq \ell\leq 2n$ by \eqref{eta2}.
By the same \eqref{varpi-5}, $\varpi_{k}$, $0\leq k\leq 2n$, only requires $\lambda_{n_1},\ldots,\lambda_{n_j}$ for $n_1,\ldots,n_j\leq 2n+2$. Thus by \eqref{lam}, it only requires $f^{(\ell)}(\gamma)$, $\ell=2,\ldots,2n+2$. Consequently, $\varpi_{k}$ for $0\leq k\leq 2n$ are independent of $y$, and the terms
$\sum_{k=0}^{2n}\varpi_ky^k$ 
are the corresponding terms in the Taylor expansion of $g(x)\frac{dx}{dy}$.

This implies that
\begin{align}\label{N-1}
\frac{d^i}{dy^i}\Big(g(x)\frac{dx}{dy}\Big)=\sum_{k=0}^{2n-i}\frac{(k+i)!}{k!}\varpi_{k+i}y^k+R_i(y).
\end{align}
where $R_i(y)$ is the remainder term.
We want to show
\begin{align}\label{N-2}
R_i(y)\ll O_n\Big(U|y|^{2n+1-i}\Big(\frac{1}{M^{2n+1}}+\frac{1}{N^{2n+1}}\Big)\Big).
\end{align}
In the following we will only consider the case of $i=1$. Other cases are similar.

From \eqref{g^i} and \eqref{eta'-m-1}
\begin{align}\label{eta'-m-1-1}
\frac{dg}{dx}&
=\sum_{k=0}^{2n-1}{\eta_k^{(1)}}'y^k+O_n\Big(U|y|^{2n}\Big(\frac{1}{NM^{2n}}+\frac{1}{N^{2n+1}}\Big)\Big),
\end{align}
where
\begin{align*}
{\eta_k^{(1)}}'=\frac{(k+1)!}{k!}\eta_{k+1}+\sum_{m=1}^{k}\eta_m\sum_{j=1}^{k+1-m}\frac{C_{kmj}}{\lambda_2^j}
\sum_{{\mbox{\tiny$\begin{array}{c}
3\leq n_1,\ldots,n_j\leq 2n+3\\n_1+\cdots+n_j=k+1-m+2j \end{array}$}}}\lambda_{n_1}\cdots\lambda_{n_j}
\end{align*}
By \eqref{dx/dy-1} we see that
\begin{align*}
\frac{d^2x}{dy^2}=2\lambda_2\frac{1}{f'}-2\lambda_2y\frac{f''}{(f')^2}\frac{dx}{dy}=\frac{1}{y}\Big(\frac{dx}{dy}-\frac{1}{2\lambda_2}\Big(\frac{dx}{dy}\Big)^3f''\Big).
\end{align*}
From last equation, \eqref{f^i(x)-2} for $i=2$, and  \eqref{dx/dy} we see that $\frac{d^2x}{dy^2}$ can be expressed as  power series of $y$, i.e.
\begin{align}\label{N-3}
\frac{d^2x}{dy^2}=\sum_{k=0}^{2n-1}\rho_k^{(1)}y^{k}+O\Big(\frac{|y^{2n}|}{M^{2n+1}}\Big)
\end{align}
where
\begin{align*}\label{N-4}
\rho_k^{(1)}= \sum_{j=1}^{k+1}\frac{C_{kj}}{\lambda_2^{j}}\sum_{{\mbox{\tiny$\begin{array}{c}
3\leq n_1,\ldots,n_j\leq 2n+3\\n_1+\cdots+n_j=k+1+2j \end{array}$}}}\lambda_{n_1}\cdots\lambda_{n_j}.
\end{align*}

With these preparations, we compute
\begin{align*}
\frac{d}{dy}\Big(g(x)\frac{dx}{dy}\Big)=\frac{dg}{dx}\Big(\frac{dx}{dy}\Big)^2+g(x)\frac{d^2x}{dy^2}.
\end{align*}
By \eqref{eta'-m-1-1}, \eqref{dx/dy}, \eqref{g} and \eqref{N-3} we get
\begin{eqnarray}\label{N-5}
\frac{d}{dy}\Big(g(x)\frac{dx}{dy}\Big)
&=&\sum_{m=0}^{2n-1}\varpi^{(1)}_my^{m}+O_n\Big(U|y|^{2n}\Big(\frac{1}{M^{2n+1}}+\frac{1}{N^{2n+1}}\Big)\Big)
\end{eqnarray}
with
\begin{eqnarray*}
\varpi^{(1)}_m=\sum_{{\mbox{\tiny$\begin{array}{c}
k,\ell_1,\ell_2\geq0\\k+\ell_1+\ell_2=m \end{array}$}}}{\eta_k^{(1)}}'\rho_{\ell_1}\rho_{\ell_2}+
\sum_{{\mbox{\tiny$\begin{array}{c}
k,\ell\geq0\\k+\ell=m \end{array}$}}}{\eta_k^{(1)}}'\rho_{\ell}^{(1)}.
\end{eqnarray*}

For $0\leq m\leq 2n-1$, $\varpi^{(1)}_m$ as above involves $\eta_{\ell}$, $0\leq \ell\leq 2n$, and $\lambda_{n_1},\ldots,\lambda_{n_j}$ with $3\leq n_1,\ldots,n_j\leq 2n+2$, and hence is independent of $y$. Consequently, the terms for $0\leq m\leq2n-1$ in \eqref{N-5} are terms in the Taylor expansion of $\frac{d}{dy}\Big(g(x)\frac{dx}{dy}\Big)$. Comparing this with the Taylor terms in \eqref{N-1} for $i=1$, we conclude that for $0\leq k\leq 2n-1$ we have 
$(k+1)\varpi_{k+1}=\varpi^{(1)}_{k}$ 
by the uniqueness of Taylor expansions. Therefore we see that \eqref{N-2} holds for $i=1$.
\end{proof}

\section{Proof of Theorem \ref{WSPI}}

Recall that \eqref{u,v-range}, we have for $\alpha\leq x\leq u$
\begin{equation*}
|f^\prime(x)|\geq \frac{T|x-\gamma|}{C_2M^2}\geq \frac{T|u-\gamma|}{C_2M^2}\geq \frac{Tr}{C^{'}_2M^2}.
\end{equation*}
Therefore by Theorem 1.1 for $n+1$ we get
\begin{eqnarray}\label{error-r}
&&
\int_{\alpha}^{u}g(x)e(f(x))dx+\int_{v}^{\beta}g(x)e(f(x))dx
\\
&=&
\Big[e(f(x)) \sum_{i=1}^{n+1}H_{i}(x)\Big]_{\alpha}^{u}+\Big[e(f(x)) \sum_{i=1}^{n}H_{i}(x)\Big]_{v}^{\beta}
+O\Big(\frac{UM^{2n+5}}{NT^{n+2}r^{n+2}} \sum_{j=1}^{[\frac{n+1}{2}]}
\frac{1}{r^j}\sum_{t=j}^{n+1-j}\frac{1}{N^{n+1-j-t}M^{t}}\Big)\nonumber
\\
&&+
O\Big(\frac{UM^{2n+4}}{T^{n+2}r^{n+2}}\Big[ \frac{1}{N^{n+1}}\Big(\frac{M}{N}+1\Big)+\sum_{j=1}^{n+1}
\frac{1}{r^j}\sum_{t=0}^{n+1-j}\frac{1}{N^{n+1-j-t}M^{t}}\Big]\Big).
\nonumber
\end{eqnarray}
Then by \eqref{estimate-r}, the error terms in \eqref{error-r} 
are
\begin{eqnarray}\label{aab}
&&
O\Big(\frac{UM^{2n+4}}{T^{n+2}N^{n+1}}\Big(\frac{M}{N}+1\Big)
\Big(\frac{1}{(\gamma-\alpha)^{n+2}}
+\frac{1}{(\beta-\gamma)^{n+2}}\Big)\Big)
\\
&+&
O\Big(\frac{UM^{2n+4}}{T^{n+2}}\sum_{j=1}^{n+1}
\Big(\frac{1}{(\gamma-\alpha)^{2n+3}}+\frac{1}{(\beta-\gamma)^{2n+3}}\Big)\Big)
+
O\Big(\frac{U}{T^{n+1}}\Big(\frac{M^{2n+2}}{N^{2n+1}}+M\Big)\Big).
\nonumber
\end{eqnarray}

By \eqref{error-r} and \eqref{aab} we only need to consider 
the remaining integral 
\begin{eqnarray}\label{mainterm5}
\int_{u}^{v}g(x)e(f(x))dx
&=&
e(f(\gamma))\int_{-r}^{r}e(\lambda_2y^2)g(x)\frac{dx}{dy}dy
\\
&=&
e(f(\gamma))\sum_{k=0}^{2n}\varpi_k\int_{-r}^{r}e(\lambda_2y^2)y^kdy+e(f(\gamma))\int_{-r}^{r}e(\lambda_2y^2)Q(y)
dy,
\nonumber
\end{eqnarray}
with $Q(y)$ defined in \eqref{Qg-y}, 
because $\varpi_k$, $0\leq k\leq 2n$, is independent of $y$.
We will now compute the first integral on the right hand side 
of \eqref{mainterm5}.

\begin{lemma}\label{inteyk}
For even $k=2j$ define
\begin{eqnarray}\label{E-phi-1}
\phi_{t}^{(k)}(y)
&=&
(-1)^{t-1}\frac{(2j-1)\cdots (2j-(2t-3))}{(4\pi \lambda_2 i)^t}
y^{2j-(2t-1)} ,\ 1\leq t\leq j+1,
\\
\label{E-phi-2}
&=&
\frac{(-1)^{j}(2j-1)!!(2t-2j-3)!!}{(4\pi i\lambda_2)^ty^{2(t-j)-1}},\ j+2\leq t\leq n+1,
\end{eqnarray}
where the numerator equals 1 for $t=1$. Then 
\begin{equation}\label{y^k}
\int_{-r}^{r}e(\lambda_2 y^2)y^{k}dy
=\Big[e(\lambda_2 y^2) \sum_{i=1}^{j}\phi_{i}^{(k)}(y)\Big]_{-r}^{r}
+\frac{(-1)^{j}(2j-1)!!}{(4\pi i\lambda_2)^j} 
\int_{-r}^{r}e(\lambda_2 y^2)dy.
\end{equation}
\end{lemma}

\begin{proof}
By \eqref{E-phi-1} and \eqref{E-phi-2} we know 
\begin{eqnarray}\label{N-T-2}
\phi_1^{(k)}(y)
=
\frac{y^{k-1}}{4\pi \lambda_2 i },\ \phi_t^{(k)}(y)=-\frac{(\phi_{t-1}^{(k)}(y))'}{4\pi \lambda_2 iy}\ \text{for $2\leq t\leq n+1$}
\end{eqnarray}
for $2n\geq k$. Applying integration by parts $j$ times we get
\begin{eqnarray}\label{Fst}
\int_{-r}^{r}e(\lambda_2 y^2)y^kdy
&=&
\Big[e(\lambda_2 y^2) \sum_{i=1}^{j}\phi_{i}^{(k)}(y)\Big]_{-r}^{r}-
\int_{-r}^{r}e(\lambda_2y^2)(\phi_{j}^{(k)}(y))'dy.
\end{eqnarray}
From \eqref{E-phi-1} for $t=j$ we get 
\begin{eqnarray}\label{E-phi-3}
\phi_{j}^{(k)}(y)=(-1)^{j-1}\frac{(2j-1)!!  y}{(4\pi \lambda_2 i)^j},
\ (\phi_{j}^{(k)}(y))'=(-1)^{j-1}\frac{(2j-1)!!}{(4\pi \lambda_2 i)^j}.
\end{eqnarray}
Substituting \eqref{E-phi-3} into \eqref{Fst} we prove 
\eqref{y^k}.
\end{proof}

The integral on the right side of \eqref{y^k} can be expressed in terms of the probability integral (cf. Gradshteyn
and Ryzhik \cite{GradRzhk} 8.251.1)
\begin{align*}
\Phi(x)=\frac{1}{\sqrt{\pi}}\int_{0}^{x^2}\frac{e^{-t}}{\sqrt{t}}dt.
\end{align*}
In fact,
\begin{equation}\label{n-1}
\int_{-r}^{r}e(\lambda_2 y^2)dy
=\frac{2}{\sqrt{2\pi\lambda_2} }\int_{0}^{x}e^{it^2}dt
=\frac{e(\frac{1}{8})}{\sqrt{2\lambda_2}}
\Phi\Big(\frac{x}{e(1/8)}\Big)
\end{equation}
for $x=r\sqrt{2\pi\lambda_2}$, by \cite{GradRzhk} 8.256.1. 
An asymptotic expansion of \eqref{n-1} is given by
\cite{GradRzhk} 8.254. Therefore
\begin{eqnarray}\label{n-2}
\int_{-r}^{r}e(\lambda_2 y^2)dy
&=&
\frac{e(\frac{1}{8})}{\sqrt{2\lambda_2}}\Big(1-\frac{e^{ix^2}e(\frac{1}{8})}{x\sqrt{\pi}}\Big[\sum_{k=0}^{d}(-1)^k\frac{(2k-1)!!}{(2x^2/i)^k}+O(x^{-2d-2})\Big]\Big)
\\
&=&
\frac{e(\frac{1}{8})}{\sqrt{f''(\gamma)}}+\frac{e(\lambda_2 r^2)}{2\pi i\lambda_2r}\Big(1+\sum_{k=2}^d\frac{(2k-3)!!}{(4\pi i\lambda_2)^{k-1}r^{2k-2}}\Big)
+O\Big(\frac{1}{\lambda_2^{d+1}r^{2d+1}}\Big).
\nonumber
\end{eqnarray}
By \eqref{L-f''} we see that the error term in \eqref{n-2} is
\begin{align}\label{n-3}
O\Big(\frac{M^{2d+2}}{T^{d+1}r^{2d+1}}\Big).
\end{align}

Substituting \eqref{n-2} and \eqref{n-3} into \eqref{y^k}, we get for $k=2j$
\begin{eqnarray*}
&&
\int_{-r}^{r}e(\lambda_2 y^2)y^kdy
\\
&=&
\frac{(-1)^{j}(2j-1)!!}{(4\pi i\lambda_2)^j}\frac{e(\frac{1}{8})}{\sqrt{f''(\gamma)}}
\nonumber
+
2e(\lambda_2 r^2)\frac{(-1)^{j}(2j-1)!!}{(4\pi i\lambda_2)^j}\frac{1}{4\pi i\lambda_2r}\Big(1+\sum_{k=2}^d\frac{(2k-3)!!}{(4\pi i\lambda_2)^{k-1}r^{2k-2}}\Big)
\nonumber
\\
&&+
\Big[e(\lambda_2 y^2) \sum_{i=1}^{j}\phi_{i}^{(k)}(y)\Big]_{-r}^{r}
+O\Big(\frac{M^{2d+2j+2}}{T^{d+j+1}r^{2d+1}}\Big).
\nonumber
\end{eqnarray*}
By \eqref{E-phi-1} and \eqref{E-phi-2}  we see that the second term  is
exactly equal to $\sum_{i=j+1}^{j+d}\phi_{i}^{(k)}(r)$.
Therefore
\begin{eqnarray}\label{n-4-2}
\int_{-r}^{r}e(\lambda_2 y^2)y^kdy
&=&
\frac{(-1)^{j}(2j-1)!!}{(4\pi i\lambda_2)^j}\frac{e(\frac{1}{8})}{\sqrt{f''(\gamma)}}
+\Big[e(\lambda_2 y^2) \sum_{i=1}^{j+d}\phi_{i}^{(k)}(y)\Big]_{-r}^{r}
\\
&&+
O\Big(\frac{M^{2d+2j+2}}{T^{d+j+1}r^{2d+1}}\Big).
\nonumber
\end{eqnarray}
For $j\leq n+1$,  we take $d=n+1-j$ in \eqref{n-4-2}. 

Now we need a bound for $\varpi_k$:  
\begin{equation}\label{varpi1}
\varpi_k
\ll
\sum_{l=0}^{k}\frac{U}{N^l}\sum_{j=1}^{k-l}\frac{M^{2j}}{T^j}\frac{T^j}{M^{k-l+2j}}
\ll
\frac{U}{M^k}\sum_{l=0}^{k}\Big(\frac{M}{N}\Big)^{l}
\ll
U\Big(\frac{1}{N^{k}}+\frac{1}{M^{k}}\Big)\ \text{for}\ 1\leq k\leq 2n
\end{equation}
by \eqref{varpi-5} and \eqref{U-f'}--\eqref{U-g}. 
Then by \eqref{varpi1} and \eqref{n-4-2} we have
\begin{eqnarray}\label{End-y^k}
\sum_{k=0}^{2n}\varpi_k\int_{-r}^{r}e(\lambda_2y^2)y^kdy
&=&
\frac{e(\frac{1}{8})}{\sqrt{f''(\gamma)}}\sum_{j=0}^{n}\varpi_{2j}\frac{(-1)^{j}(2j-1)!!}{(4\pi i\lambda_2)^j}+e(\lambda_2r^2)\Big[\sum_{k=0}^{2n}\varpi_{k}\sum_{i=1}^{n+1}\phi_{i}^{(k)}(y)\Big]_{-r}^{r}
\\
&&+
O\Big(\sum_{j=0}^{n}U\Big(\frac{M}{N}+1\Big)^{2j}\frac{M^{2n-2j+4}}{T^{n+2}r^{2n-2j+3}}\Big),
\nonumber
\end{eqnarray}
where we added in the terms for odd $k$ which are zero anyway.

Let us turn to the second integral on the right hand side of \eqref{mainterm5}.

\begin{lemma}\label{int-r+rO}
With $Q(y)$ as in \eqref{Qg-y} define 
\begin{equation}\label{N-T-Q-1}
\psi_{1}(y)=\frac{Q(y)}{(4\pi i\lambda_2)y}
\ \ \text{and}\ \ 
\psi_{k}(y)=-\frac{\psi_{k-1}^\prime(y)}{(4\pi i\lambda_2)y}
\ \ \text{for}\ \ 2\leq k\leq n+3.
\end{equation}
Then 
\begin{equation}\label{End-Q}
\int_{-r}^{r}Q(y)e(\lambda_2 y^2)dy
=
\Big[e(\lambda_2 y^2) \sum_{k=1}^{n+1}\psi_{k}(y)\Big]_{-r}^{r}
+
O\Big(\frac{U}{T^{n+1}}\Big(\frac{M^{2n+2}}{N^{2n+1}}+M\Big)\Big).
\end{equation}
\end{lemma}

\begin{proof}
From \eqref{Qg-y} and \eqref{g-dx/dy}  we get
\begin{align}\label{Up-Q}
Q(y)
\ll_n U|y|^{2n+1}\Big(\frac{1}{N^{2n+1}}+\frac{1}{M^{2n+1}}\Big).
\end{align}
 Similarly by \eqref{Qg-y} and \eqref{g-dx-1} we see that for $1\leq t\leq n+2$
\begin{align*}
Q^{(t)}(y)=\frac{d^t}{dy^t}\Big(g(x)\frac{dx}{dy}\Big)-\sum_{k=0}^{2n-t}\frac{(k+t)!}{k!}\varpi_{k+t}y^k.
\end{align*}
Therefore by \eqref{g-dx-1} for $1\leq t\leq n+2$
\begin{align}\label{Up-Q'}
Q^{(t)}(y)
&\ll U|y|^{2n+1-t}\Big(\frac{1}{N^{2n+1}}+\frac{1}{M^{2n+1}}\Big).
\end{align}

Next we choose a real number $\delta\asymp\lambda_2^{-1/2}$ such that $r/\delta$ is a power of $2$. The total variation of 
$Q(y)$ on $[-\delta,\delta]$ is 
\begin{align}\label{V-Q}
V(Q(y))
\ll
\int_{-\delta}^{\delta}\Big|\frac{dQ}{dy}\Big|dy
\ll
U|\delta|^{2n+1}\Big(\frac{1}{N^{2n+1}}+\frac{1}{M^{2n+1}}\Big).
\end{align}
From \eqref{Qg-y}, \eqref{Up-Q}, \eqref{V-Q},
by the Second Derivative Test (see Lemma 5.1.3 of 
\cite{Hxly}, p.88), we have
\begin{eqnarray}\label{delta}
\int_{-\delta}^{\delta}Q(y)e(\lambda_2 y^2)dy
&\ll&
\frac{\max\limits_{-\delta\leq y\leq \delta}|Q(y)|+V(Q(y))}{\sqrt{\lambda_2}}
\ll \frac{U|\delta|^{2n+1}}{\sqrt{\lambda_2}}\Big(\frac{1}{N^{2n+1}}+\frac{1}{M^{2n+1}}\Big)\\
&\ll& \frac{U}{T^{n+1}}\Big(\frac{M^{2n+2}}{N^{2n+1}}+M\Big).\nonumber
\end{eqnarray}

We split the range $\delta\leq y\leq r,-r\leq y\leq-\delta$ into intervals of the form
$t\leq y\leq2t,-2t\leq y\leq-t $. By integration by parts we have
\begin{align}\label{Q1}
&\int_{t}^{2t}Q(y)e(\lambda_2 y^2)dy=\Big[e(\lambda_2 y^2) \sum_{i=1}^{n+1}\psi_{i}(y)\Big]_{t}^{2t}-\int_{t}^{2t}e(\lambda_2 y^2)\psi_{n+1}^\prime(y)dy.
\end{align}
From \eqref{Up-Q'} we see that for $1\leq k\leq n+3$
\begin{eqnarray}\label{pi}
\psi_k(y)
&=&
\sum_{i=0}^{k-1}c_{ki}\frac{Q^{(i)}(y)}{(4\pi i\lambda_2)^ky^{2k-1-i}}\ll\frac{U|y|^{2n+2-2k}}{\lambda_2^k}\Big(\frac{1}{M^{2n+1}}+\frac{1}{N^{2n+1}}\Big)
\\
&\ll&
\frac{Ut^{2n+2-2k}}{\lambda_2^k}\Big(\frac{1}{M^{2n+1}}+\frac{1}{N^{2n+1}}\Big),
\nonumber
\end{eqnarray}
since $t\leq |y|\leq 2t$.
Hence
\begin{equation*}
\psi_{n+1}'(y)=-4\pi i\lambda_2 y\psi_{n+2}(y)\ll \frac{U}{\lambda_2^{n+1}t}\Big(\frac{1}{M^{2n+1}}+\frac{1}{N^{2n+1}}\Big),
\end{equation*}
and the total variation of $\psi_{n+1}'(y)$ on $[t,2t]$ is
\begin{eqnarray*}
V(\psi_{n+1}'(y))
&=&
\int_{t}^{2t}|\psi_{n+1}''(y)|dy=\int_{t}^{2t}|-4\pi i\lambda_2\psi_{n+2}(y)+(4\pi i\lambda_2 y)^2\psi_{n+3}(y)|dy
\\
&\ll&
\frac{U}{\lambda_2^{n+1}t}\Big(\frac{1}{M^{2n+1}}+\frac{1}{N^{2n+1}}\Big).
\end{eqnarray*}
By the First Derivative Test (Lemma 5.1.2 of \cite{Hxly}, p.88) we get
$$
\int_{t}^{2t}e(\lambda_2 y^2)\psi_{n}^\prime(y)dy
\ll
\frac{\max\limits_{t\leq y\leq 2t}\{\psi_n'(y)\}+V(\psi_n'(y))}{\lambda_2t}
\ll
\frac{U}{\lambda_2^{n+2}t^2}\Big(\frac{1}{N^{2n+1}}+\frac{1}{M^{2n+1}}\Big).
$$
Note that $\lambda_2\gg T/M^2$ and by \eqref{Q1} we get
$$
\int_{t}^{2t}Q(y)e(\lambda_2 y^2)dy
=
\Big[e(\lambda_2 y^2) \sum_{k=1}^{n+1}\psi_{k}(y)\Big]_{t}^{2t}
+
O\Big(\frac{UM^2}{T^{n+2}t^2}\Big(\frac{M^{2n+2}}{N^{2n+1}}+M\Big)\Big).
$$

Summing over ranges with $t=2^k\delta,k=0,1,2,\ldots$, we get
\begin{eqnarray}\label{pi3}
\int_{\delta}^{r}Q(y)e(\lambda_2 y^2)dy
&=&
\Big[e(\lambda_2 y^2) \sum_{k=1}^{n+1}\psi_{k}(y)\Big]_{\delta}^{r}
+O\Big(\frac{UM^2}{T^{n+2}\delta^2}\Big(\frac{M^{2n+2}}{N^{2n+1}}+M\Big)\sum_{k\geq 1}
\frac{1}{2^{2k}}\Big)
\\
&=&
\Big[e(\lambda_2 y^2) \sum_{k=1}^{n}\psi_{k}(y)\Big]_{\delta}^{r}+O\Big(\frac{U}{T^{n+1}}\Big(\frac{M^{2n+2}}{N^{2n+1}}+M\Big)\Big).
\nonumber
\end{eqnarray}
Similarly
\begin{equation}\label{pi4}
\int_{-r}^{-\delta}Q(y)e(\lambda_2 y^2)dy
=
\Big[e(\lambda_2 y^2) \sum_{k=1}^{n}\psi_{k}(y)\Big]_{-r}^{-\delta}
+O\Big(\frac{U}{T^{n+1}}\Big(\frac{M^{2n+2}}{N^{2n+1}}+M\Big)\Big).
\end{equation}
Now by \eqref{pi3}, \eqref{pi4} and \eqref{delta} we see that
\begin{eqnarray}\label{intQ3terms}
\int_{-r}^{r}Q(y)e(\lambda_2 y^2)dy
&=&
\Big[e(\lambda_2 y^2) \sum_{k=1}^{n}\psi_{k}(y)\Big]_{-r}^{r}
-
\Big[e(\lambda_2 y^2) \sum_{k=1}^{n}\psi_{k}(y)\Big]_{-\delta}^{\delta}
\\
&&+
O\Big(\frac{U}{T^{n+1}}\Big(\frac{M^{2n+2}}{N^{2n+1}}+M\Big)\Big).
\nonumber
\end{eqnarray}
By \eqref{pi} we have the following trivial estimates
\begin{eqnarray}\label{2ndterm<<}
\Big[e(\lambda_2 y^2) \sum_{k=1}^{n+1}\psi_{k}(y)\Big]_{-\delta}^{\delta}
&\ll&
\sum_{k=1}^{n+1}\frac{U\delta^{2n+2-2k}}{\lambda_2^k}\Big(\frac{1}{M^{2n+1}}+\frac{1}{N^{2n+1}}\Big)
\\
&\ll&
\frac{U}{T^{n+1}}\Big(\frac{M^{2n+2}}{N^{2n+1}}+M\Big).
\nonumber
\end{eqnarray}
Therefore by \eqref{intQ3terms} and \eqref{2ndterm<<} we 
prove the lemma.
\end{proof}

Now by \eqref{mainterm5}, \eqref{End-y^k}, \eqref{End-Q}  we have
\begin{eqnarray}\label{u-v}
&&
\int_{u}^{v}g(x)e(f(x))dx
\\
&=&
e(f(\gamma))\frac{e(\frac{1}{8})}{\sqrt{f''(\gamma)}}\sum_{j=0}^{n}\varpi_{2j}\frac{(-1)^{j}(2j-1)!!}{(4\pi i\lambda_2)^j}
\nonumber
\\
&&+
e(f(\gamma))\Big[e(\lambda_2y^2)\sum_{k=0}^{2n}\varpi_{k}\sum_{i=1}^{n+1}\phi_{i}^{(k)}(y)\Big]_{-r}^{r}
+e(f(\gamma))\Big[e(\lambda_2 y^2) \sum_{k=1}^{n+1}\psi_{k}(y)\Big]_{-r}^{r}
\nonumber
\\
&&+
O\Big(\sum_{j=0}^{n}U\Big(\frac{M}{N}+1\Big)^{2j}\frac{M^{2n-2j+4}}{T^{n+2}r^{2n-2j+3}}\Big)
+
O\Big(\frac{U}{T^{n+1}}\Big(\frac{M^{2n+2}}{N^{2n+1}}+M\Big)\Big).
\nonumber
\end{eqnarray}
By  \eqref{estimate-r}, the first $O-$term in \eqref{u-v} is
\begin{eqnarray}\label{u-v-end}
&\ll&
\frac{U}{T^{n+1}}\Big(\frac{M^{2n+1}}{N^{2n}}+M\Big)
+\frac{UM^{2n+4}}{T^{n+2}}\Big(\frac{1}{(\gamma-\alpha)^{2n+3}}+\frac{1}{(\beta-\gamma)^{2n+3}}\Big)
\\
&&+
\frac{UM^{2n+4}}{T^{n+2}N^{2n}}\Big(\frac{1}{(\gamma-\alpha)^{3}}+\frac{1}{(\beta-\gamma)^{3}}\Big).
\nonumber
\end{eqnarray}

Let
\begin{align*}
&F(y)=f(\gamma)+\lambda_2y^2=f(x),\ G(y)=g(x)\frac{dx}{dy},\\
&\theta_1(y)=\frac{G(y)}{2\pi iF^\prime(y)},\ \theta_i(y)=-\frac{\theta^{\prime}_{i-1}(y)}{2\pi iF^\prime(y)}\ \text{for}\ 2\leq i\leq n+1.
\end{align*}
Now we want to show the following two equalities,
\begin{eqnarray}\label{N-T-5}
\Big[e(F(y)) \sum_{i=1}^{n+1}\theta_{i}(y)\Big]_{-r}^{r}
&=&
e(f(\gamma))\Big[e(\lambda_2y^2)\sum_{k=0}^{2n}\varpi_{k}\sum_{i=1}^{n+1}\phi_{i}^{(k)}(y)\Big]_{-r}^{r}
\\
&&+
e(f(\gamma))\Big[e(\lambda_2 y^2) \sum_{k=1}^{n+1}\psi_{k}(y)\Big]_{-r}^{r}
\nonumber
\\
\label{N-T-6}
\Big[e(F(y)) \sum_{i=1}^{n+1}\theta_{i}(y)\Big]_{-r}^{r}
&=&
\Big[e(f(x)) \sum_{i=1}^{n+1}H_{i}(x)\Big]_{u}^{v},
\end{eqnarray}
where $H_{i}(x)$ is defined in \eqref{the2-1}. By 
\eqref{Qg-y} we see that
$$
G(y)=\sum_{k=0}^{2n}\varpi_ky^k+Q(y),
\hspace{3mm}
\theta_1(y)=\sum_{k=0}^{2n}\varpi_k\frac{y^{k-1}}{4\pi i\lambda_2}+\frac{Q(y)}{4\pi i\lambda_2y}.
$$
By \eqref{N-T-2} and \eqref{N-T-Q-1}  we see that for $1\leq i\leq n+1$
\begin{align*}
\theta_i(y)=\sum_{k=0}^{2n}\varpi_k\phi_{i}^{(k)}(y)+\psi_{i}(y).
\end{align*}
Therefore \eqref{N-T-5} is true.

To prove \eqref{N-T-6} we see that
\begin{align*}
F^\prime(y)=f^\prime(x)\frac{dx}{dy}\ \text{and}\ G(y)=g(x)\frac{dx}{dy},
\end{align*}
and hence
\begin{align*}
\theta_1(y)=\frac{G(y)}{2\pi iF^\prime(y)}=\frac{g(x)}{2\pi if^\prime(x)}=H_{1}(x),
\end{align*}
\begin{align*}
\theta_i(y)=-\frac{\theta^\prime_{i-1}(y)}{2\pi iF^\prime(y)}=-\frac{H'_{i-1}(x)}{2\pi if^\prime(x)}=H_{i}(x)\ \text{for}\ 2\leq i\leq n+1.
\end{align*}
Now by induction  \eqref{N-T-6} follows from the last two formulas and the correspondence between  $y=r$ and $x=v$, and between $y=-r$ and $x=u$.

Then we can conclude from \eqref{u-v}, \eqref{u-v-end}, 
\eqref{N-T-5} and \eqref{N-T-6} that 
\begin{eqnarray}\label{end}
\int_{u}^{v}g(x)e(f(x))dx
&=&
\frac{e\Big(f(\gamma)+\frac{1}{8}\Big)}{\sqrt{f''(\gamma)}}\Big(g(\gamma)+\sum_{j=1}^{n}\varpi_{2j}\frac{(-1)^{j}(2j-1)!!}{(4\pi i\lambda_2)^j}
\Big)
\\
&&+\Big[e(f(x)) \sum_{i=1}^{n+1}H_{i}(x)\Big]_{u}^{v}+O\Big(\frac{U}{T^{n+1}}\Big(\frac{M^{2n+2}}{N^{2n+1}}+M\Big)\Big)\nonumber
\\
&&+U\Big(\frac{M}{N}+1\Big)^{2n}\frac{M^{2n+4}}{T^{n+2}}\Big(\frac{1}{(\gamma-\alpha)^{2n+3}}+\frac{1}{(\beta-\gamma)^{2n+3}}\Big).
\nonumber
\end{eqnarray}

At last, we consider the whole integral \eqref{Ige(f)}. 
Since
$\int_\alpha^\beta=\int_\alpha^u+\int_u^v+\int_v^\beta$,
then by \eqref{aab} and  \eqref{end}, we can prove  
\eqref{thm3}.
\qed

{\it Acknowledgments.} The authors would like to thank a 
referee for a detailed reading and thoughtful comments. 
Gratitude is also due to Xiumin Ren who gave the authors 
many helpful suggestions. 

\bibliographystyle{amsplain}

\end{document}